\documentclass[11pt]{amsart}
\usepackage[linkcolor=blue, citecolor = blue]{hyperref}
\hypersetup{colorlinks=true}

\usepackage{amsthm,amsfonts,amsmath,amssymb,amscd} 

\usepackage{verbatim}
\usepackage{float}
\usepackage{wrapfig}
\usepackage{mathtools}
\usepackage[color=orange!20, linecolor=orange]{todonotes}

\usepackage{enumitem}

\usepackage{ytableau}

\usepackage{tikz}
\usetikzlibrary{decorations.markings}
\tikzstyle{vertex}=[circle, draw, inner sep=0pt, minimum size=4pt]

\usetikzlibrary{arrows,automata}
\usepackage{tikz, tikz-3dplot, pgfplots}
\usetikzlibrary[positioning,patterns]

\newtheorem{theorem}{Theorem}[section]
\newtheorem{proposition}[theorem]{Proposition}
\newtheorem{lemma}[theorem]{Lemma}
 
\theoremstyle{definition}
\newtheorem{corollary}[theorem]{Corollary}
\newtheorem{definition}[theorem]{Definition}
\newtheorem{example}[theorem]{Example}
\newtheorem{problem}[theorem]{Problem}

\theoremstyle{remark}
\newtheorem{remark}[theorem]{Remark}

\newcounter{x}
\newcounter{y}
\newcounter{z}

\newcommand\xaxis{210}
\newcommand\yaxis{-30}
\newcommand\zaxis{90}

\newcommand\topside[3]{
  \fill[fill=white, draw=black,shift={(\xaxis:#1)},shift={(\yaxis:#2)},
  shift={(\zaxis:#3)}] (0,0) -- (30:1) -- (0,1) --(150:1)--(0,0);
}

\newcommand\leftside[3]{
  \fill[fill=cyan, draw=black, shift={(\xaxis:#1)},shift={(\yaxis:#2)},
  shift={(\zaxis:#3)}] (0,0) -- (0,-1) -- (210:1) --(150:1)--(0,0);
}

\newcommand\rightside[3]{
  \fill[fill=lightgray, draw=black,shift={(\xaxis:#1)},shift={(\yaxis:#2)},
  shift={(\zaxis:#3)}] (0,0) -- (30:1) -- (-30:1) --(0,-1)--(0,0);
}
\newcommand\rrightside[3]{
  \fill[fill=white, draw=black,shift={(\xaxis:#1)},shift={(\yaxis:#2)},
  shift={(\zaxis:#3)}] (0,0) -- (30:1) -- (-30:1) --(0,-1)--(0,0);
}

\newcommand\cube[3]{
  \topside{#1}{#2}{#3} \leftside{#1}{#2}{#3} \rightside{#1}{#2}{#3}
}
\newcommand\ccube[3]{
  \topside{#1}{#2}{#3} \leftside{#1}{#2}{#3} %\rightside{#1}{#2}{#3}
}

\newcommand\planepartition[1]{
 \setcounter{x}{-1}
  \foreach \a in {#1} {
    \addtocounter{x}{1}
    \setcounter{y}{-1}
    \foreach \b in \a {
      \addtocounter{y}{1}
      \setcounter{z}{-1}
      \foreach \c in {0,...,\b} {
        \addtocounter{z}{1}
      \ifthenelse{\c=0}{\setcounter{z}{-1},\addtocounter{y}{0}}{
        \cube{\value{x}}{\value{y}}{\value{z}}}
      }
    }
  }
}

\newcommand\pplanepartition[1]{
 \setcounter{x}{-1}
  \foreach \a in {#1} {
    \addtocounter{x}{1}
    \setcounter{y}{-1}
    \foreach \b in \a {
      \addtocounter{y}{1}
      \setcounter{z}{-1}
      \foreach \c in {0,...,\b} {
        \addtocounter{z}{1}
      \ifthenelse{\c=0}{\setcounter{z}{-1},\addtocounter{y}{0}}{
        \ccube{\value{x}}{\value{y}}{\value{z}}}
      }
    }
  }
}
\title[Enumeration of plane partitions by descents]{
Enumeration of plane partitions by descents  
} 

\author[Damir Yeliussizov]{Damir Yeliussizov}
\address{KBTU, Almaty, Kazakhstan}
\email{\href{mailto:yeldamir@gmail.com}{yeldamir@gmail.com}}

\begin{document}

\begin{abstract}
We study certain bijection between plane partitions and $\mathbb{N}$-matrices. 
As applications, we prove a Cauchy-type  
identity for generalized dual Grothendieck polynomials. We introduce two statistics on plane partitions, whose generating functions are similar to classical MacMahon's formulas; one of these statistics is equidistributed with the usual volume. We also show natural connections with the longest increasing subsequences of words. 
\end{abstract}

\maketitle

%\tableofcontents

\section{Introduction}

One of the most successful bijections in combinatorics is the {\it Robinson-Schensted-Knuth} (RSK) correspondence. It has many important properties and applications, see \cite[Ch.~7 \& notes]{sta}. In particular, the RSK gives a bijection between {\it $\mathbb{N}$-matrices} %of nonnegative integers 
and {\it plane partitions}. 

In this paper, we study a different (yet somewhat analogous) bijection between these sets. 
A map from plane partitions to matrices is quite transparent; 
it is described by recording the so-called {\it descent level sets} (or can be viewed by projecting corners in $\mathbb{R}^3$ presentation). 

This approach gives a number of interesting consequences. We obtain the following results:

\begin{itemize}
\item A multivariate identity via certain {\it descent} statistics %(see \eqref{des}) 
of plane partitions. %(Theorem~\ref{?}). 
The identity  
becomes a source of several subsequent formulas. 
\item Generating functions for two volume-type statistics. We call them the {\it up-hook } and {\it corner volumes}. The formulas are analogous to classical MacMahon's identities. 
Interestingly, the up-hook volume is equidistributed with the usual volume of plane partitions.
\item A Cauchy-type identity for generalized {\it dual symmetric Grothendieck polynomials} which are $K$-theoretic inhomogeneous deformations of Schur functions. 
\item A Frobenius-type identity for {\it strict tableaux} which are generalizations of the standard Young tableaux (SYT). We show that these objects are naturally related to the {\it longest increasing subsequences} of words. 
\item Monotonicity and bounds for certain descent enumeration functions of plane partitions. 
\end{itemize}

\section{Preliminaries}
\subsection{Partitions} A {\it partition} is a sequence $\lambda = (\lambda_{1}, \ldots, \lambda_{\ell})$ of positive integers $\lambda_1 \ge \cdots \ge \lambda_{\ell}$, where $\ell(\lambda) = \ell$ is the {\it length} of $\lambda$. Every partition $\lambda$ can be represented as the {\it Young diagram} $\{(i,j): i \in [1, \ell], j\in [1,\lambda_{i}], (i,j) \in \mathbb{N}^2 \}$.

\subsection{Plane partitions}
A {\it plane partition} $\pi = (\pi_{i j})_{i,j \ge 1}$ is a two-dimensional array of nonnegative integers with finitely many nonzero entries such that
$$
\pi_{i\, j} \ge \pi_{i+1\, j}, \quad \pi_{i\, j} \ge \pi_{i\, j +1}, \quad \text{ for all } i,j \ge 1.
$$

We denote by $|\pi| := \sum_{i,j} \pi_{i j}$ the {\it volume} (or size, or total weight) of $\pi$. 

By default we ignore zero entries of plane partitions. The {\it shape} of $\pi$ denoted by $\mathrm{sh}(\pi)$ is the partition whose Young diagram is $\{(i,j) :  \pi_{ij} > 0\}$. 
It is useful to view plane partitions as a pile of cubes in $\mathbb{R}^3$ as in Fig.~\ref{fig0}. They can be bounded by a box.  
Let $\mathrm{PP}(k,n, m)$ be the set of plane partitions that fit in the box $k \times n \times m$, i.e. the length of the first row is at most $k$, the length of the first column is at most $n$, and the first entry is at most $m$. Let also $\mathrm{PP}'(k, n, m)$ be the set of plane partitions whose base shape is {\it exactly} $k \times n$, i.e. the first row length is $k$, the first column length is $n$, and the first entry is still at most $m$.  

\begin{figure}
\ytableausetup{aligntableaux = bottom}
\begin{ytableau}
 4 & {3} & {2} \\
 {3} & 3 & {1} \\  %\textcolor{blue}
 %{2} & {2}
\end{ytableau}
\qquad
\begin{tikzpicture}[scale = 0.35]
\planepartition{
{4,3,2},
{3,3,1}}
\end{tikzpicture}
%\qquad
%\begin{align*}
%\end{align*}
\caption{A plane partition $\pi \in \mathrm{PP}(3,2,4)$ with $|\pi| = 16$, $\mathrm{sh}(\pi) = (3,3)$, and its boxed presentation.
%and $|\pi|_p = 3 + 3 + 2 + 2 + 2 + 1 = 13$
}\label{fig0}
\end{figure}
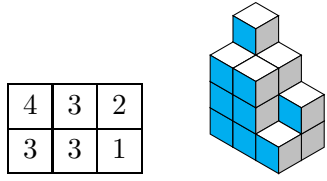

We have classical MacMahon's formula \cite{macmahon}:
\begin{align}
\sum_{\pi\, \in\, \mathrm{PP}(k, n, m)} q^{|\pi|} = \prod_{i = 1}^{k} \prod_{j = 1}^{n} \prod_{\ell = 1}^{m} \frac{1 - q^{i + j + \ell - 1}}{1 - q^{i + j + \ell - 2}}
\end{align}
which for $k,n,m \to \infty$ gives the generating function
\begin{align}
\sum_{\pi} q^{|\pi|} &= \prod_{n = 1}^{\infty} \frac{1}{(1 - q^n)^n}.
\end{align}
For $q \to 1$, we have MacMahon's famous theorem on boxed plane partitions 
\begin{align}\label{mm}
|\mathrm{PP}(k, n, m)| = \prod_{i = 1}^{k} \prod_{j = 1}^{n} \prod_{\ell = 1}^{m} \frac{i + j +\ell - 1}{i + j + \ell - 2} 
\end{align}
We refer to \cite[Ch.7]{sta} for enumerative aspects of plane partitions and to an excellent more recent survey \cite{krat}. 

\subsection{Schur polynomials}
The {\it Schur polynomials} $\{s_{\lambda}\}$ can be defined as follows:
$$
s_{\lambda}(x_1, \ldots, x_m) := \sum_{\pi\, \in\, \mathrm{SPP}(\lambda)} \prod_{i = 1}^{m} x_i^{a_i(\pi)},
$$ 
where $\mathrm{SPP}(\lambda)$ is the set of {\it column-strict plane partitions} of shape $\lambda$ (i.e. $\pi_{i j} > \pi_{i+1\, j}$) with the largest entry at most $m$, %$x^T = \prod_{i = 1}^{m} x_i^{a_i}$ 
and $a_i(\pi)$ is the number of entries $i$ in $\pi$.

The following identities are known:
\begin{align}\label{qschur}
\sum_{\pi\, \in\, \mathrm{PP}(k, n, m)} q^{|\pi|} = q^{-n \binom{k +1}{2}} s_{\rho}(q, q^2, \ldots, q^{n + m}),
\end{align}
where $\rho = (k, \ldots, k) := (k^n)$ is a rectangular partition; for $q = 1$ this gives %(e.g. \cite{st})
\begin{align}
|\mathrm{PP}(k, n, m)| = s_{\rho}(1^{n+m}),
\end{align}
where %$s_{\lambda}$ is the Schur polynomial and 
$(1^n) = (1, \ldots, 1)$ repeated $n$ times.  

\subsection{Dual Grothendieck polynomials}
These functions are certain $K$-theoretic deformations of Schur polynomials. They were first explicitly described and studied in \cite{lp}. 
\begin{definition}  
The {\it dual symmetric Grothendieck polynomials} $\{ g_{\lambda} \}$ can be defined as follows
$$
g_{\lambda}(x_1, \ldots, x_m) := \sum_{\mathrm{sh}(\pi) = \lambda} \prod_{i = 1}^{m} x_i^{c_i(\pi)},
$$
where the sum runs over plane partitions $\pi$ (with largest entry at most $n$) of shape $\lambda$, and %$x^T = \prod_{i = 1}^n x_i^{c_i}$ such that 
$c_i(\pi)$ is the number of columns of $\pi$ containing $i$.
\end{definition}

The following properties are well known \cite{lp}: $g_{\lambda}$ is a symmetric polynomial, $g_{\lambda} = s_{\lambda} + \text{lower degree terms}$, and hence $\{g_{\lambda} \}$ forms a basis of the ring $\Lambda$ of symmetric functions.  
More properties of these functions can be found in \cite{lp, dy, dy2}.

From the definition it is not difficult to obtain the following formulas
$$
|\mathrm{PP}(k,n,m)| = g_{\rho}(1^{m+1}), \qquad |\mathrm{PP}'(k,n,m)| = g_{\rho}(1^{m}),
$$
where again $\rho = (k^n)$. They are a special case of more general identities that we prove in this paper, e.g. we will prove formulas similar to \eqref{qschur}.

\section{A bijection between plane partitions and $\mathbb{N}$-matrices}
An {\it $\mathbb{N}$-matrix} is a matrix of nonnegative integers with only finitely many nonzero elements. Define the map $\Phi : \{\text{plane partitions}\} \to \{\mathbb{N}\text{-matrices}\}$ as follows. Given a plane partition $\pi$, define the {\it descent level sets} 
\begin{align}\label{ldes}
D_{i \ell} := \{ j : \pi_{i j} = \ell  > \pi_{i+1\, j} \},
\end{align}
i.e. $D_{i \ell}$ is the set of column indices of the entry $\ell$ in the $i$th row of $\pi$ that are strictly larger than the entry below. Let $d_{i \ell} := |D_{i \ell}|$ and define the matrix $D := (d_{i \ell})_{i, \ell \ge 1}$.
Set 
\begin{align}\label{piq}
\Phi(\pi) = D.
\end{align}

\begin{example}\label{ex1} When $\pi$ is the plane partition shown below, ~ %Consider $\Phi: \pi \mapsto D$
\begin{center}
{\small
$\Phi : $
\ytableausetup{aligntableaux = center}
\begin{ytableau}
 4 & {4} & {2} \\
 {4} & 2 & {1} \\  %\textcolor{blue}
 {2} & {2}
\end{ytableau}
$~\longmapsto
\left(
\begin{matrix}
0 & 1 & 0 & 1\\
1 & 0 & 0 & 1\\
0 & 2 & 0 & 0\\
%1 & 1 & 0\\
\end{matrix}
\right)
$
}
\end{center}
%The second row is $(q_{2,1}, q_{2,2}, q_{2,3}) = (1,0,2)$ as in the first row of the plane partition there is one entry $2$ larger than the entry $1$ below, in the second row 
\end{example}

Alternatively, we can view the map $\Phi$ geometrically: present $\pi$ as a pile of cubes in $\mathbb{R}^3$; 
mark the corners 
\begin{tikzpicture}[scale = 0.25]
\pplanepartition{{1}}
\end{tikzpicture}
on the surface of $\pi$ with $\bullet$ as in Fig.~\ref{figa}. Then $d_{i \ell}$ is the number of marks with the coordinates $x = i$ and $z = \ell$. 

\begin{figure}
\begin{tikzpicture}[scale = 0.40]
\planepartition{{4,4,2},{4,2,1},{2,2}}
\draw[thick, dashed,->] (0,4) -- (0,5);
\node at (0.5,5) {$z$};
\draw[thick, dashed,->] (-2.6,-1.5) -- (-3.6, -2.1); %(30:-1);
\node at (-4,-2.1) {$x$};
\draw[thick, dashed,->] (2.6,-1.5) -- (3.6, -2.1); %(30:-1);
\node at (4,-2.1) {$y$};

\foreach \c in {1,...,3} {
  \foreach \d in {1,...,4}{
	\rrightside{\c}{-3}{\d};
  }
}
\draw[thick, dashed,->] (-2.6-2.7,-1.5+1.5) -- (-3.6-2.7, -2.1+1.5); %(30:-1);
\node at (-4-2.7,-2.1+1.5) {$i$};

\draw[thick, dashed,->] (0-2.6,4+1.5) -- (0-2.6,5+1.5);
\node at (0.5-2.7,5+1.5) {$\ell$};

\node at (1.4,0.25) {{\scriptsize {$\bullet$}}};
\draw[dashed, gray] (1.4, 0.2) to (-3.5, 3);
\node at (-3.65,3.25) {{\scriptsize $1$}};
\node at (-3.65+0.2,3.25-0.25) {{\scriptsize $\bullet$}};

\node at (-1.2,-0.25) {{\scriptsize {$\bullet$}}};
\draw[dashed, gray] (1.4-2.65, 0.2-0.45) to (-3.5-2.65+1, 3-0.45-0.6);
\node at (-3.5-2.65+1-0.3, 3-0.45-0.6+0.25) {{\scriptsize $2$}};
\node at (-3.5-2.65+1-0.05, 3-0.45-0.6+0.05) {{\scriptsize $\bullet$}};

\node at (-2.1,0.25) {{\scriptsize {$\bullet$}}};

\node at (-1.2,2.75) {{\scriptsize {$\bullet$}}};
\draw[dashed, gray] (1.4-2.6, 0.2+2.5) to (-3.5-2.6+1.75, 3+2.5-1);
\node at (-3.5-2.6+1.75, 3+2.5-1) {{\scriptsize {$\bullet$}}};
\node at (-3.5-2.6+1.75-0.25, 3+2.5-1+0.25) {{\scriptsize {$1$}}};

\node at (0.5,2.75) {{\scriptsize {$\bullet$}}};
\draw[dashed, gray] (1.4-0.9, 0.2+2.5) to (-3.5-0.9+0.95, 3+2.5-0.5);
\node at (-3.5-0.9+0.95, 3+2.5-0.5) {{\scriptsize {$\bullet$}}};
\node at (-3.5-0.9+0.95-0.25, 3+2.5-0.5+0.25) {{\scriptsize {$1$}}};

\node at (-3.5-0.9+0.95-0.25-3, 3+2.5-0.5+0.25-1) {{\scriptsize {$D = (d_{i \ell})$}}};

\node at (0.5,-1.25) {{\scriptsize {$\bullet$}}};
\draw[dashed, gray] (1.4-0.9, 0.2-1.5) to (1.4-2.65, 0.2-0.45);
\node at (-3.5-2.65+1-0.25+1-0.1, 3-0.45-0.6+0.25-1+0.55) {{\scriptsize $1$}};
\node at (-3.5-2.65+1-0.25+1+0.08, 3-0.45-0.6+0.25-1+0.3) {{\scriptsize $\bullet$}};
\end{tikzpicture}
\caption{A geometric view of the map $\Phi$.}\label{figa}
\end{figure}
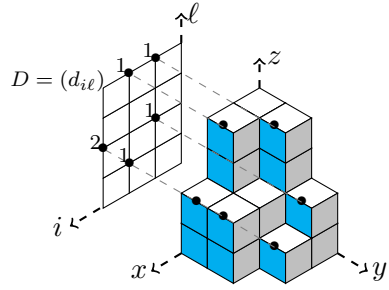
%\end{remark}

\begin{theorem}\label{bij}
The map $\Phi$ defines a bijection between the set $\mathrm{PP}(\infty, n, m)$ of plane partitions with at most $n$ rows and entries not exceeding $m$ and $\mathbb{N}$-matrices with $n$ rows and $m$ columns.
\end{theorem}
\begin{proof}
It is obvious that for $\pi \in \mathrm{PP}(\infty, n, m)$, the map $\Phi$ produces an $n \times m$ $\mathbb{N}$-matrix. 

Let us describe the inverse map $\Phi^{-1}$. Given an $n \times m$ $\mathbb{N}$-matrix $D$, we show how to uniquely reconstruct a plane partition $\pi  \in \mathrm{PP}(\infty, n, m)$ such that $\Phi(\pi) = D$. We will build $\pi$ sequentially by scanning the columns of $D$ starting from the last one. 

Let us first show how to {\it add} element $\ell$ in $i$-th row of some plane partition $\pi'$. To do so, find the first available column of $\pi'$ whose length is less than $i$, then add the elements $\ell$ in this column so that its length becomes $i$. For example, suppose $\pi'$ already had the following shape and we want to add $\ell$ in the 3rd row:
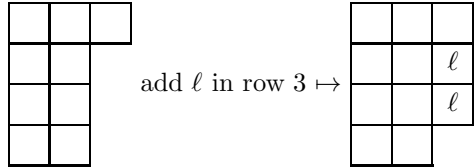
\begin{center}
{\small
\begin{ytableau}
 ~ & ~ & {~} \\
 {~} & ~  \\  
 {~} & {~} \\
 ~ & ~
\end{ytableau}
add $\ell$ in row $3$ $\mapsto$
\begin{ytableau}
 ~ & {~} & {~} \\
 ~ & ~ & {\ell} \\  
 {~} & {~} & {\ell} \\
 ~ & ~
\end{ytableau}
}
\end{center}

Let us initially set $\pi = \varnothing$. For each $\ell = m, m-1, \ldots, 1$ consider the $\ell$-th column $(d_{1 \ell }, \ldots, d_{n \ell })^T$ of $D$. 
%starting from the last one we do the following. 
For each $i = n, n-1, \ldots, 1$ add $\ell$ in $i$-th row of $\pi$ exactly $d_{i \ell}$ times. It is easy to see that after this procedure, the resulting $\pi$ is actually a plane partition with at most $n$ rows, the entries not exceeding $m$, and %the resulting plane partition 
that $\pi$ has the desired property, $\Phi(\pi) = D$.
\end{proof}

\begin{example}\label{ex3}
Let us see how $\Phi^{-1}$ applies to the matrix from Example~\ref{ex1}. We have: 
\begin{center}
{\small
$ %\Phi^{-1} : 
D=
\left(
\begin{matrix}
0 & 1 & 0 & 1\\
1 & 0 & 0 & 1\\
0 & 2 & 0 & 0\\
%1 & 1 & 0\\
\end{matrix}
\right) %\longmapsto 
$
}
\end{center}
and hence $\pi$ is constructed as follows: 

{\small
%\begin{center}
add column 4 of $D$: \quad add $4$ in row $2$:  
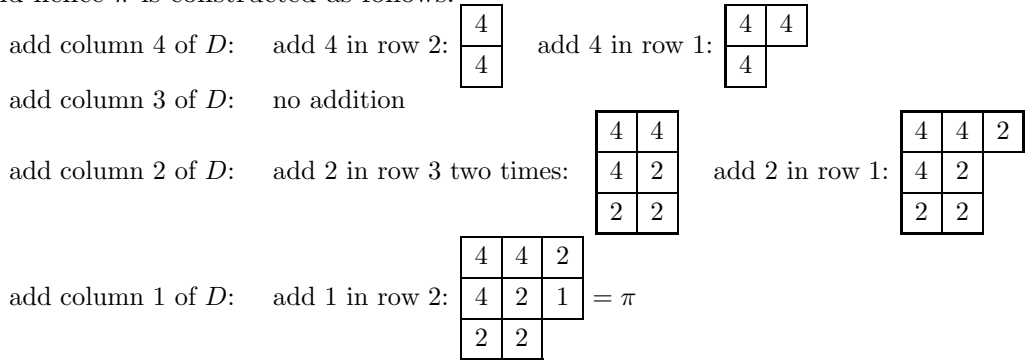
\begin{ytableau}
 4  \\
 {4}   %\textcolor{blue}
\end{ytableau}
\quad add $4$ in row $1$:  
\begin{ytableau}
 4  & 4 \\
 {4}   %\textcolor{blue}
\end{ytableau}

add column 3 of $D$: \quad no addition

add column 2 of $D$: \quad add $2$ in row $3$ two times:\quad  
\begin{ytableau}
 4  & 4\\
 {4}  & 2\\ %\textcolor{blue}
 2 & 2
\end{ytableau}
\quad add $2$ in row $1$: 
\begin{ytableau}
 4  & 4 & 2\\
 {4}  & 2\\ %\textcolor{blue}
 2 & 2
\end{ytableau}

add column 1 of $D$: \quad add $1$ in row $2$:
\begin{ytableau}
 4  & 4 & 2\\
 {4}  & 2 & 1\\ %\textcolor{blue}
 2 & 2
\end{ytableau} $ = \pi$
%\end{center}
}
\end{example}

Let us point out a property of the bijection $\Phi$ that we showed in \cite{dy4} for probabilistic applications. %It is to see from the proof that in fact $\Phi$ defines a bijection between $\mathrm{PP}(\infty,n,m)$ and $\mathbb{N}$-matrices with $n$ rows and $m$ columns. %$\mathrm{BM}(a,b,c)$. If $\pi \in \mathrm{PP}(a,b,c)$, then $D = \Phi(\pi) \in \mathrm{BM}(a,b,c)$.
The length of the first row of $\pi$ is equal to the maximal weight down-right path (i.e. with steps $(i,j) \to (i+1, j), (i, j+1)$) in $D = \Phi(\pi)$ from $(1,1)$ to $(n,m)$. %Hence, $\Phi$ defines a bijection between $\mathrm{PP}(k,n,m)$ and $\mathbb{N}$-matrices whose maximal weight up-right path from $(1,1)$ does not exceed $k$. 
This is similar to the classical RSK account, see \cite{pak, sag}. 

\section{A multivariate identity}

The following notion (related to the previous notion of descent level sets) will be important for us and used throughout. For a plane partition $\pi$, define the {\it descent set} 
\begin{align*}
\mathrm{Des}(\pi) := \{(i,j) : \pi_{i j} > \pi_{i+1\, j}\}.
\end{align*}
Denote also $\mathrm{des}(\pi) = |\mathrm{Des}(\pi)|$. %\footnote{This notion is different from descent statistics defined for standard Young tableaux in \cite{?}.} 

As a first application of the bijection $\Phi$ described in previous section we prove the following identity which is a source of several formulas given later.
\begin{theorem}\label{mult}
The following formula holds:
\begin{align*}
\sum_{\pi} \prod_{(i,j)\, \in\, \mathrm{Des}(\pi)}  x_i z_{\pi_{i j}} = \prod_{i = 1}^{n} \prod_{\ell = 1}^{m} \frac{1}{1 - x_i z_\ell},
\end{align*}
where the sum runs over plane partitions $\pi \in \mathrm{PP}(\infty, n,m)$.
\end{theorem}
\begin{proof}
Let us write the generating function
\begin{align}\label{cau0}
\sum_{D} \prod_{i = 1}^{n} \prod_{\ell = 1}^{m} (x_i z_{\ell})^{d_{i \ell}} = \prod_{i = 1}^n \prod_{\ell = 1}^{m} \frac{1}{1- x_i z_\ell}
\end{align}
over all $\mathbb{N}$-matrices $D = (d_{i \ell})$ with $n$ rows and $m$ columns, $1 \le i \le m, 1 \le \ell \le n$. Applying the bijection $\Phi^{-1}$, the l.h.s of \eqref{cau0} rewrites as follows
\begin{align*}
\sum_{D} \prod_{i = 1}^{n} \prod_{\ell = 1}^{m} (x_i z_\ell)^{d_{i \ell}} = \sum_{\pi} \prod_{i = 1}^{n} \prod_{\ell = 1}^{m} (x_i z_\ell)^{|\{j\,  :\, \pi_{i j} = \ell > \pi_{i+1\, j} \}|} = \sum_{\pi} \prod_{(i, j) \in \mathrm{Des}(\pi)}  x_{i} z_{\pi_{i j}}, %= \sum_{\lambda} g_{\lambda}(\mathbf{x}; \mathbf{z}),
\end{align*}
where the sum is over plane partitions $\pi \in \mathrm{PP}(\infty, n, m)$.
\end{proof}

\section{Generalized dual Grothendieck polynomials}
\begin{definition}
%For a boxed plane partition $\pi \in \mathrm{PP}(a,b,c)$ of shape $\lambda$, and 
Define the polynomials $g_{\lambda}(\mathbf{x}; \mathbf{z})$ in two sets of variables  $\mathbf{x} = (x_1, \ldots, x_{n})$ and $\mathbf{z} = (z_1, \ldots, z_m)$ indexed by partitions $\{\lambda\}$ as follows:
\begin{align}\label{gdef}
g_{\lambda}(\mathbf{x}; \mathbf{z}) := \sum_{\mathrm{sh}(\pi) = \lambda} \,\prod_{(i,j)\, \in\, \mathrm{Des}(\pi)}  x_i z_{\pi_{i j}},
\end{align}
where the sum runs over plane partitions $\pi \in \mathrm{PP}(\infty, n,m)$ of shape $\lambda$.
Note that $g_{\lambda}(\mathbf{x}; \mathbf{z}) = 0$ unless $\ell(\lambda) \le n$.
\end{definition}

%Now we summarize some properties of these polynomials.
\begin{theorem}
%Let $\mathbf{x}_c = (x_1, \ldots, x_c), \mathbf{y}_{b} = (y_1, \ldots, y_b)$. 
The polynomials $g_{\lambda}(\mathbf{x}; \mathbf{z})$ satisfy the following properties:

(i) We have the expansion
\begin{align*}
g_{\lambda}(\mathbf{x}; \mathbf{z}) = \mathbf{x}^{\lambda} s_{\lambda}(\mathbf{z}) + \text{lower degree terms},
\end{align*}
where $\mathbf{x}^{\lambda} = x_1^{\lambda_1} \cdots x_{n}^{\lambda_n}$. %In particular, $\{g_{\lambda}(\mathbf{x}; \mathbf{z}) \}$ are linearly independent (over $\mathbb{C}$). 
Moreover, each monomial $\mathbf{x}^{\alpha} \mathbf{z}^{\beta}$ in $g_{\lambda}$ has $|\alpha| = |\beta|$.

\vspace{0.5em}

(ii) We have the formula %The formula \eqref{gdef} can be rewritten as
\begin{align}\label{gdc}
g_{\lambda}(\mathbf{x} ; \mathbf{z}) = \sum_{\mathrm{sh}(\pi) = \lambda}\, \mathbf{x}^{d(\pi)} \mathbf{z}^{c(\pi)},\quad \text{ where }\quad \mathbf{x}^{d(\pi)} = \prod_{i = 1}^{n} x_i^{d_i(\pi)},\quad  \mathbf{z}^{c(\pi)} = \prod_{j = 1}^{m} z_j^{c_j(\pi)} %= \sum_{{\mathrm{sh}(\pi) = \lambda}} \prod_{i = 1}^{n} \prod_{j = 1}^{m} (x_i z_j)^{m_{i j}},
\end{align}
where the sum runs over plane partitions $\pi \in \mathrm{PP}(\infty, n,m)$ of shape $\lambda$; 
$c_j(\pi)$ is the number of columns of $\pi$ containing $j$ and $d_i(\pi)$ is the number of entries in $i$th row of $\pi$ that are strictly larger than the entry below. 

\vspace{0.5em}

(iii) $g_{\lambda}(1^n ; \mathbf{z}) = g_{\lambda}(\mathbf{z})$ coincides with the dual Grothendieck polynomial for $\ell(\lambda) \le n$. %and $0$ otherwise.

\vspace{0.5em}

(iv) The following Cauchy-type identity holds
\begin{align}\label{cauchy}
\sum_{\ell(\lambda) \le n} g_{\lambda}(\mathbf{x} ; \mathbf{z}) =  \prod_{i = 1}^{n} \prod_{j = 1}^{m} \frac{1}{1 - x_i z_j}
\end{align}
\end{theorem}
\begin{proof}
(i) A plane partition $\pi$ contributed to the top degree component when the set $\mathrm{Des}(\pi)$ contains all cells of $\pi$, i.e. when $\pi$ is a column-strict plane partition. In this case, the $\mathbf{x}$ monomial will be $\mathbf{x}^{\lambda}$ corresponding to the shape of $\pi$. Summing over such $\pi$, we get %and $\mathbf{z}$ monomial will sum up to 
the Schur polynomial $s_{\lambda}(\mathbf{z})$. %The linear independence of the family $\{g_{\lambda}(\mathbf{x}; \mathbf{z}) \}$ follows then from linear independence of Schur polynomials and 
It is clear by definition that monomial expansions have equal $x$ and $z$ degrees. %clearly by definition.

(ii) The expansion \eqref{gdc} easily follows from the definition and recounting degrees.

(iii) This immediately follows from definition of dual Grothendieck polynomials and the formula \eqref{gdc} when setting $\mathbf{x} = 1^n$.

(iv) Using the definition \eqref{gdef} and Theorem~\ref{mult} we have
$$
\sum_{\lambda} g_{\lambda}(\mathbf{x}; \mathbf{z}) = \sum_{\pi} \prod_{(i, j) \in \mathrm{Des}(\pi)}  x_{i} z_{\pi_{i j}}  =  \prod_{i = 1}^{n} \prod_{j = 1}^{m} \frac{1}{1 - x_i z_j}
$$
\end{proof}

\begin{corollary} The following identity holds for dual Grothendieck polynomials
\begin{align}\label{gl}
\sum_{\ell(\lambda) \le n} g_{\lambda}(\mathbf{z}) = \prod_{i = 1}^{m} \frac{1}{(1 - z_i)^{n}}
\end{align}
\end{corollary}
\begin{proof}
Plug $\mathbf{x} = 1^{n}$ in \eqref{cauchy} and use (iii).
\end{proof}
\begin{remark}
Our proof of this identity as well as \eqref{cauchy} is essentially bijective, using the map $\Phi$ that preserves weights of $g_{\lambda}$.  
%\end{remark}
%\begin{remark}
The identity \eqref{gl} has an algebraic proof using the family  of stable Grothendieck polynomials $\{G_{\lambda} \}$ dual to $\{g_{\lambda} \}$ and the Cauchy identity. It uses non-obvious properties of certain automorphisms of $G_{\lambda}$ discussed in \cite{buch}. See \cite{dy2, dy3}.
\end{remark}

\begin{remark}
Since $\{g_{\lambda} \}$ is a deformation of Schur polynomials, the formula \eqref{gl} can be compared to the following identity for Schur functions \cite{macdonald}
$$
\sum_{\ell(\lambda) \le n} s_{\lambda}(z_1, \ldots, z_n) = \prod_{1 \le i \le n}\frac{1}{1 - z_i} \prod_{1 \le i < j \le n} \frac{1}{1 - z_i z_j}.
$$
\end{remark}

\begin{remark}
The polynomials $g_{\lambda}(\mathbf{x} ; \mathbf{z})$ are in fact a translated version of the {\it refined} dual stable Grothendieck polynomials $\tilde{g}_{\lambda}(\mathbf{t}; \mathbf{z})$ introduced in \cite{ggl}. This can be seen via the formula \eqref{gdc}; namely, we have $g_{\lambda}(\mathbf{x}; \mathbf{z}) = \mathbf{x}^{\lambda} \tilde{g}_{\lambda}(\mathbf{x}^{-1}; \mathbf{z})$. %In particular, 
This implies that $g_{\lambda}(\mathbf{x}; \mathbf{z})$ is {\it symmetric} in the $\mathbf{z}$ variables.
\end{remark}

\begin{problem}
How can the bijection $\Phi$ be extended to prove the {\it Cauchy identity} for symmetric Grothendieck polynomials? Namely, the identity 
$$
\sum_{\lambda} G_{\lambda}(\mathbf{x}) g_{\lambda}(\mathbf{z}) = \prod_{i,j} \frac{1}{1 - x_i z_j}.
$$
%where $G_{\lambda}$ is the symmetric (or stable) Grothendieck polynomial, see \cite{buch}.
Note that our approach works for $\mathbf{x} = 1^n$ for which we obtained the identity \eqref{gl}. More generally, the problem is about a bijective proof of the {\it skew} Cauchy identity proved in \cite{dy2}
$$
\sum_{\lambda} G_{\lambda/\!\!/\mu}(\mathbf{x}) g_{\lambda/\nu}(\mathbf{z}) = \prod_{i,j} \frac{1}{1 - x_i z_j}  \sum_{\kappa} G_{\nu/\!\!/\kappa}(\mathbf{x}) g_{\mu/\kappa}(\mathbf{z}),
$$
which would be a generalization of skew RSK for Schur functions given in \cite{ss}.
\end{problem}

\section{Two statistics on plane partitions}
\subsection{Up-hooks}
For a plane partition $\pi$ %, recall that %define the {\it descent set} 
%\begin{align*}
%\mathrm{Des}(\pi) := \{(i,j) : \pi_{i, j} > \pi_{i+1, j}\}.
%\end{align*}
and a cell $(i,j)$, %of a plane partition $\pi$, 
define the {\it up-hook} 
$$
\mathrm{uh}(i,j) := \pi_{i j} + i - 1.
$$
When $\pi$ presented as a pile of cubes, up-hooks look as follows (here $\mathrm{uh}(i,j) = 5 + 4 - 1$):
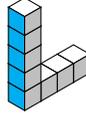
\begin{center}
\begin{tikzpicture}[scale = 0.25]
\planepartition{{1},{1},{1},{5}}
\end{tikzpicture}
\end{center}
%\begin{definition}
%For a plane partition $\pi$, 
Now, define the {\it up-hook volume} of $\pi$ as:
$$
|\pi|_{{uh}} := \sum_{(i,j)\, \in\, \mathrm{Des}(\pi)} \mathrm{uh}(i,j).%(\pi_{i j} + i - 1)  %i,j\, :\, \pi_{i,j} > \pi_{i+1,j}
$$
%\end{definition}

\begin{example}
Consider the plane partition %$\pi$ given as follows:
\begin{center}
$\pi = $ \begin{ytableau}
 4 & {\bf 4} & {2} \\
 {\bf 4} & 2 & {\bf 2} \\  %\textcolor{blue}
 {\bf 2} & {\bf 2} 
\end{ytableau}
\end{center}
Here the bold entries correspond to $\mathrm{Des}(\pi)$ cells. We have e.g. $\mathrm{uh}(2,1) = 4 + 2 - 1 = 5$ and 
$$ 
|\pi|_{uh} = (4 + 1 - 1) + (4 + 2 - 1) + (2 + 2 - 1) + (2 + 3 - 1) + (2 + 3 - 1) = 20.
$$
Note that $|\pi| = 22 \ne |\pi|_{uh}$. %and $|\pi|_c = 4 + 3 + 4 + 2 + 2 + 2 = 17$.
\end{example}

\begin{theorem} The following identity holds
\begin{align}\label{teq}
\sum_{\pi\, \in\, \mathrm{PP}(\infty, n, m)} t^{\mathrm{des}(\pi)} q^{|\pi|_{uh}}= \prod_{i = 1}^{m} \prod_{j = 1}^{n} \frac{1}{1 - t\, q^{i + j - 1}}.
\end{align}
\end{theorem}

\begin{proof}
The identity follows from the multivariate formula in Theorem \ref{mult} by setting $x_i = t q^{i-1}$ and $z_{j} = q^{j}$ which gives $$\prod_{(i,j) \in \mathrm{Des}(\pi)} x_i z_{\pi_{i j}} = t^{\mathrm{des(\pi)}} q^{|\pi|_{uh}}$$
and hence the needed identity.
\end{proof}
\begin{remark}
Since the proof of Theorem~\ref{mult} is bijective, we have a bijective proof of \eqref{teq}.
\end{remark}

Letting $m,n \to \infty$ we obtain  
the following {\it equidistribution} of volume and up-hook-volume statistics.
\begin{corollary}[Equidistribution of volume and up-hook-volume]
We have:
$$
\sum_{\pi} q^{|\pi|_{uh}} = \prod_{k = 1}^{\infty} \frac{1}{(1 - q^k)^k} = \sum_{\pi} q^{|\pi|}. 
$$
Moreover, the pair statistics $(\mathrm{des}(\pi), |\pi|_{uh})$ and $(\mathrm{tr}(\pi), |\pi|)$ are also equidistributed:
$$
\sum_{\pi} t^{\mathrm{des}(\pi)} q^{|\pi|_{uh}} = \prod_{k= 1}^{\infty} \frac{1}{(1 - t q^k)^k} = \sum_{\pi} t^{\mathrm{tr}(\pi)} q^{|\pi|},
$$ 
where $\mathrm{tr}(\pi) = \sum_{i} \pi_{i i}$ is the {\it trace} of a plane partition.
\end{corollary}

Similarly, for $m \to \infty$ or $n \to \infty$ we obtain the following identities.
\begin{corollary} We have
\begin{align*}
\sum_{\pi\, \in\, \mathrm{PP}(\infty, \infty, m)} q^{|\pi|_{uh}} &= %\prod_{i = 1}^{m} 
\prod_{j = 1}^{\infty} {(1 - q^{j})^{-\min(j,m)}}\\
\sum_{\pi\, \in\, \mathrm{PP}(\infty, n, \infty)} q^{|\pi|_{uh}} &= \prod_{i = 1}^{\infty} 
{(1 - q^{i})^{-\min(i,n)}}
\end{align*}
\end{corollary}

\begin{remark}
It would be interesting to see a direct combinatorial argument why the statistics $|\cdot|_{uh}$ and $|\cdot|$ are equidistributed (and when paired with des and trace). 
\end{remark}

\begin{remark}
The map $\pi_{i j} \mapsto \mathrm{uh}(i,j) = \pi_{i j} + i - 1$ for all $i,j$ defines a bijection between (weak) {\it reverse plane partitions} (RPP which may contain zero entries) and semistandard Young tableaux (SSYT) (see e.g. \cite[Ch.~7.22]{sta}). This bijection gives the famous formula  
\begin{align}\label{hilg}
\sum_{\mathrm{sh}(\pi) = \lambda} q^{|\pi|} =q^{|\lambda| - b(\lambda)} s_{\lambda}(1,q,q^2, \ldots) = q^{|\lambda|} \prod_{(i,j) \in \lambda} \frac{1}{1 - q^{\,h_{\lambda}(i,j)}},
\end{align}
where the sum runs over RPP $\pi$ of shape $\lambda$, $b(\lambda) = \sum (i-1)\lambda_i$ and $h_{\lambda}(i,j) = \lambda_i - i + \lambda'_j - j + 1$ is the {\it hook length} of $(i,j)$ in $\lambda$. A bijective proof of \eqref{hilg} is known as the {\it Hillman-Grassl correspondence} \cite{hg}, see also \cite[Ch.~7.21-22]{sta}. It would be interesting to see  if  there is any analogous formula when $|\cdot|$ is replaced with the $|\cdot|_{uh}$ statistic.
\end{remark}

\subsection{Corner volume}
For a plane partition $\pi$, define the \textbf{{\it corner volume}} statistic $|\pi|_c$ as 
$$ 
|\pi|_c := \sum_{(i,j)\, \in\, \mathrm{Des}(\pi)} \pi_{i,j}
$$
Equivalently, $|\pi|_{c} = \sum_{i \ge 1} i\, c_i$, where $c_i$ is the number of columns of $\pi$ containing $i$.

%Since $n$ cannot be unbounded 
Our next result can be viewed as the corner volume counterpart of the classical MacMahon's formula.
\begin{theorem}\label{mtm}
Let $\rho = (k^n)$. The following identities hold:
\begin{align}
\sum_{\pi\, \in\, \mathrm{PP}(k,\, n, \, m)} q^{|\pi|_c} &= s_{\rho}(1^n, q, q^2, \ldots, q^{m})\label{q1} \\ %=\sum_{\lambda \subset (k^\ell)} \prod_{x \in \lambda} \frac{q^{?}}{1 - q^{h(x)}} \prod_{(i,j) \in \lambda^{\vee}} \frac{\ell + i - j}{h_{i,j}}\\
\sum_{\pi\, \in\, \mathrm{PP}'(k,\, n, \, m)} q^{|\pi|_c} &= s_{\rho}(1^{n-1}, q, q^2, \ldots, q^{m}) \label{q2}\\
\sum_{\pi\, \in\, \mathrm{PP}(\infty, n, m)} q^{|\pi|_c} &= \prod_{i = 1}^{m} {(1 - q^i)^{-n}} \label{q3}
\end{align}
\end{theorem}

One can let $m \to \infty$ in all these formulas. 

For $q = 1$ in \eqref{q1} we recover MacMahon's boxed formula \eqref{mm}; from the second identity \eqref{q2} we get the formula for the number of plane partitions with the first row length $k$, the first column length $n$, and the first entry at most $m$: 
$$|\mathrm{PP}'(k,n,m)| = s_{\rho}(1^{m+n-1}) =|\mathrm{PP}(k,n,m-1)|.$$ 

To prove Theorem~\ref{mtm} we need some results on dual Grothendieck polynomials.

\begin{lemma}\label{l1}
Let $\rho = (k^{n})$. We have 
\begin{align}\label{ggg}
g_{\rho}(\mathbf{z}) = s_{\rho}(1^{n - 1}, \mathbf{z}).
\end{align}
\end{lemma}
\begin{proof}
We show that the Jacobi-Trudi formulas for both polynomials coincide for the shape $\rho$. Note first that by the classical Jacobi-Trudi formula we have
\begin{align}\label{sr}
s_{\rho}(1^{n - 1}, \mathbf{z}) = \det\left[ e_{n - i + j}(1^{n - 1}, \mathbf{z})\right]_{1 \le i, j \le k}.
\end{align}
For the polynomials $g$, we have the following Jacobi-Trudi identity %\eqref{jt} %
%we have for any $\lambda$ 
\cite{dy}
\begin{align}\label{gjt}
g_{\lambda}(\mathbf{z}) = \det\left[ e_{\lambda'_i - i + j}(1^{\lambda'_i - 1}, \mathbf{z})  \right]_{1 \le i,j \le \lambda_1}
\end{align}
which for $\lambda = \rho$ gives \eqref{sr}.
\end{proof}

\begin{remark}
This lemma together with the bijection $\Phi$ is important for various probabilistic applications particularly related to certain {\it corner} distributions of {\it lozenge tilings}, given in our recent works \cite{dy4, dy5}. %, see also \cite{dy5}.
\end{remark}

\begin{remark}
Jacobi-Trudi-type formulas for dual Grothendieck polynomials were originally presented in  \cite{sz}. A (combinatorial) proof  of \eqref{gjt} %via the Gessel-Viennot method and RSK 
is given in \cite{dy}. See also \cite{kim, ay} with more general identities for skew shapes. 
\end{remark}

\begin{remark}
We would also like to point out that dual Grothendieck polynomials of rectangular shapes %These types of 
appear in solutions of discrete {Toda} equations, see \cite{in}. 
\end{remark}

\begin{corollary}
We have the following formula
\begin{align}
\sum_{\lambda \subset (k^{n})} g_{\lambda}(\mathbf{z}) &= s_{(k^{n})}(1^{n}, \mathbf{z})\label{gx} %\\
%\sum_{\lambda \subset (k^{\ell})} g_{\lambda}(\mathbf{x})\, g_{\lambda^c}(\mathbf{y}) &= s_{(k^{\ell})}(1^{\ell - 1}, \mathbf{x}, \mathbf{y}), 
\end{align}
%where $\lambda^c = (k - \lambda_{\ell}, \ldots, k - \lambda_{1})$ is the complement partition to $\lambda$.
\end{corollary}
\begin{proof}
Let $\rho = (k^{n})$. By the combinatorial definition of $g_{\lambda}$, we have the following branching formula
$$
g_{\rho}(1,\mathbf{z}) = \sum_{\lambda \subset \rho} g_{\lambda}(\mathbf{z}).
$$
On the other hand, $g_{\rho}(1,\mathbf{z}) = s_{\rho}(1^{n}, \mathbf{z})$ by Lemma~\ref{l1}. 
\end{proof}

\begin{proof}[Proof of Theorem~\ref{mtm}]
Note that for  the principal specialization $z_i = q^{i}$ we have
$$
g_{\lambda}(q, q^2, \ldots) = \sum_{\mathrm{sh}(\pi) = \lambda} \prod_{(i,j) \in \mathrm{Des}(\pi)} q^{\pi_{i j}} =  \sum_{\mathrm{sh}(\pi) = \lambda} q^{|\pi|_c}.
$$
The identities \eqref{q1}, \eqref{q2} now follow from \eqref{gx} and \eqref{ggg}, respectively. The identity \eqref{q3} follows from \eqref{gl}.
\end{proof}

\begin{remark}
%{Partitions with parts from a set} 
By taking appropriate specializations, we can easily generalize our formulas for {\it restricted} partitions with parts from any set of nonnegative integers. %Let $S \subset \mathbb{N}$.
\end{remark}

\begin{remark}
It is not difficult to show that both of the presented statistics $|\cdot|_{uh}$ and $|\cdot|_c$ are {\it superadditive} functions on plane partitions, i.e.
$$
|\pi + \pi'|_{\alpha} \ge |\pi|_{\alpha} + |\pi'|_{\alpha}, 
$$
where $\alpha$ is $uh$ or $c$. Moreover, $|\cdot|_c$ is an {\it antinorm}, i.e. besides superadditivity we also have
\begin{itemize}
\item[(i)] $|\pi|_c \ge 0$ %for all $\pi$ 
and $|\pi|_c=0$ implies $\pi=0$
\item[(ii)] $|k\, \pi|_c=k|\pi|_c$ for $k \in \mathbb{N}$.
%\item[(iii)] $\mathrm{cvol}(\pi + \pi') \ge \mathrm{cvol}(\pi) + \mathrm{cvol}(\pi')$
\end{itemize}
However, the item (ii) fails for $|\cdot|_{uh}$ and instead we have
\begin{align*}%\label{kp}
k |\pi|_c \le |k\, \pi|_{uh} = |\pi|_{uh} + (k - 1) |\pi|_c \le k |\pi|_{uh}.
\end{align*}
Note also that we always have $|\pi|_{uh}, |\pi| \ge |\pi|_{c}$ and $|\cdot|$ is an additive norm.
\end{remark}
 
\begin{remark}
%{Symmetry classes} 
Enumeration of plane partitions for the usual volume %was successful in establishing 
has many product formulas for various {\it symmetry classes}, see \cite{sta2}. It would be interesting to see if there are any analogous generating functions 
%It is not clear to us whether any good formulas exist 
for the weights $|\cdot |_{uh}$ or $|\cdot|_c$ over some symmetric plane partitions.
\end{remark}

\section{Strict tableaux and longest increasing subsequences of words}
%Consider the following generalization of standard Young tableaux (SYT). 
Say that a plane partition $\pi$ is a {\it strict tableau} if it is filled with entries $[n] := \{1, \ldots, n \}$ such that each $i \in [n]$ appears in exactly one column of $\pi$. We then say that $[n]$ is a {\it filling} of $\pi$.
Let $\mathrm{ST}(\lambda, n)$ be the set
%$f_{\lambda}(n)$ be the number 
of strict tableaux of shape $\lambda$ with the filling $[n]$ and $f_{\lambda}(n) = |\mathrm{ST}(\lambda, n)|$. %plane partitions $\pi$ of shape $\lambda$ filled with entries $\{1, \ldots, n \}$ such that each $i \in [1,n]$ appears in exactly one column of $\pi$. 
We clearly have $f_{\lambda}(n) = 0$ unless $\lambda_1 \le n \le |\lambda|$ and $f_{\lambda}(|\lambda|) = f^{\lambda}$ is the number of standard Young tableaux (SYT) of shape $\lambda$.

Denote also by $W_{n,m}$ the set of words of length $n$ in the alphabet $[m]$.

%Another application of the bijection $\Phi$ is the following identity.

\begin{theorem}
The following identity holds:
\begin{align}\label{idd}
\sum_{\lambda \subset (n^m)} f_{\lambda}(n) = m^n.
\end{align}
Furthermore, there is a bijection between the sets $W_{n,m}$ and $\mathrm{ST}(\lambda, n)$ for $\lambda \subset (n^m)$.
%of words of length $n$ in the alphabet $[m]$ and strict tableaux of shape $\lambda \subset (n^m)$ with content $[n]$.
\end{theorem}

\begin{proof}
When we apply the bijection $\Phi$ on a strict tableau $\pi$, the resulting matrix $D$ contains a single $1$ in every of its $n$ columns, and has at most $m$ rows. Such $D$ clearly corresponds to a unique word in $W_{n,m}$ by writing down row indices of $1$'s in $D$. Conversely, it is easy to see that any word in $W_{n,m}$ converted into a matrix $D$ and applying the inverse map $\Phi^{-1}$ gives a unique element of $\mathrm{ST}(\lambda, n)$ for some $\lambda \subset (n^m)$.
\end{proof}

\begin{remark}
As it was noticed by one of the referees,  a map from strict tableaux to words  has the following simple description: For each $i \in [n]$ write down the largest row index in which $i$ appears.
\end{remark}

We now show an algebraic motivation of strict tableaux and a way to derive the formula \eqref{idd}. The numbers $f_{\lambda}(n)$ can be extracted from the so-called {\it Plancherel specialization} of the ring of symmetric functions $\Lambda$ applied to the basis $\{g_{\lambda}\}$. Namely, let $\theta : \Lambda \to \mathbb{R}[[t]]$ be the specialization (algebra homomorphism) given by $$\theta : p_1 \mapsto t \text{ and } \theta : p_k \mapsto 0 \text{ for } k \ge 2,$$ where $p_k = \sum_{i} x_i^k$ are the power sum symmetric functions which are generators of $\Lambda$.

\begin{proposition}
The following expansion holds
\begin{align}\label{gexp}
\theta(g_{\lambda}) = \sum_{n = \lambda_1}^{|\lambda|} f_{\lambda}(n) \frac{t^n}{n!}
\end{align}
\end{proposition}
\begin{proof}
Follows from the fact that for any symmetric function $g \in \Lambda$ we have:
$$
\theta(g) = \sum_{n} [x_1 \cdots x_n] g\, \frac{t^n}{n!},
$$
where $[x_1 \cdots x_n] g$ is the coefficient of $x_1 \cdots x_n$ in the monomial expansion of $g$, see \cite{ges, sta}. It is easy to see that $[x_1 \cdots x_n] g_{\lambda} = f_{\lambda}(n)$ hence giving the identity.
\end{proof}
Now, using the identity \eqref{gl} derived earlier, on the other hand we have
$$
\sum_{\ell(\lambda) \le m} \theta(g_{\lambda}) = \theta \prod_{i \ge 1} {(1 - x_i)^{-m}}= e^{m t}.
$$
Therefore, combining this with \eqref{gexp} we obtain
$$
\sum_{\ell(\lambda) \le m} \theta(g_{\lambda}) = \sum_{\ell(\lambda) \le m} \sum_{n \ge \lambda_1}^{} f_{\lambda}(n) \frac{t^n}{n!} = e^{m t} = \sum_{n} m^n \frac{t^n}{n!}
$$
which by comparing the coefficients of $t^{n}$ proves \eqref{idd}.

\begin{remark}
The identity \eqref{idd} derived from the Plancherel specialization on dual Grothendieck polynomials is an analog of the classical Frobenius-type identity 
$$
\sum_{\lambda \vdash n} f^{\lambda} s_{\lambda}(1^m)  = m^n %n! 
$$
that follows from the Plancherel specialization and $(1^{m})$ applied on Schur functions in the Cauchy identity. The bijection $\Phi$ is then an analog of the RSK.
%On Plancherel measure and specialization.
\end{remark}

\begin{remark}
Strict tableaux arise naturally as walks on the so-called {\it $\beta$-filtration} of Young's lattice, see \cite{dy2}. This filtration is an example of more general {\it dual filtered graphs} introduced and studied in \cite{pp}.
\end{remark}

For a word $w  = w_1 \cdots w_n \in W_{n, m}$, a {\it weakly increasing subsequence} is a sequence of the form
$$
w_{i_1} \le \cdots \le w_{i_k}, \quad 1 \le i_1 < \cdots < i_k \le n,
$$
where $k$ is its length. Let $L_{i}(w)$ be the length of the largest weakly increasing subsequence of $w$ using the letters $\{m-i+1,\ldots, m \}$. In particular, $L_{1}(w)$ is just the number of $m$'s in $w$ and $L_{m}(w)$ is the length of the longest weakly increasing subsequence of $w$.

For a word $w  = w_1 \cdots w_n \in W_{n,m}$ let $D(w) = (d_{ij})$ be an $m \times n$ $01$-matrix defined as $d_{w_i, i} = 1$ for all $i \in [n]$ and all other elements of $D$ are zero. Note that $D$ has single $1$ in each column.

\begin{theorem}
Let $w \in W_{n,m}$ and $\pi = \Phi^{-1}(D(w))$. 
Then the shape of $\pi$ is $(L_{m}(w), \ldots, L_{1}(w))$.
\end{theorem}
\begin{proof}
Let $\lambda_1 \ge \cdots \ge \lambda_m \ge 0$ be the shape of $\pi$. We need to show that $\lambda_{k} = L_{m - k + 1}(w)$ for all $k \in [m]$. 

Take any down-right lattice path $\Pi$ from $(k,1)$ to $(m,n)$. Then it is not difficult to see that the descent level sets $D_{i \ell}$ for $(i,\ell) \in \Pi$ are pairwisely disjoint, note that $i \ge k$. %Using this and  since $i \ge k$ for all $(i,\ell) \in \Pi$, 
Hence we obtain
\begin{equation}\label{lal}
\sum_{(i,\ell) \in \Pi} d_{i \ell} = \sum_{(i,\ell)} |\{ j : \pi_{i j} = \ell > \pi_{i + 1 j} \}| \le \lambda_k.
\end{equation}
On the other hand, suppose the $k$-th row of $\pi$ has entries $(\ell_1 \ge \cdots \ge \ell_{s})$ where $s = \lambda_k$. Assume the entries $\ell_1, \ldots, \ell_s$ end in the rows $i_1 \ge \cdots \ge i_s$ of $\pi$. Then there is a down-right path $\Pi$ from $(k,1)$ to $(m,n)$ that contains all points $(i_s, \ell_s), \ldots, (i_1, \ell_1)$. The weight of any such path is at least 
$\sum_{j} d_{i_j \ell_j} \ge s = \lambda_k$. Combining this with the inequality \eqref{lal} we obtain that
$$
\lambda_k = \max_{\Pi} \sum_{(i,\ell) \in \Pi} d_{i \ell},
$$ 
where the maximum is taken over down-right paths $\Pi$ from $(k,1)$ to $(m,n)$. 
Finally, by translating $D(w)$ to $w$, one can easily see that $L_{m - k + 1}(w)$ is also the maximal weight of an down-right path from $(k,1)$ to $(m,n)$ in $D$.  
\end{proof}

\begin{example}
Let $w = 132434 \in W_{6,4}$. We have 
$$
{\small
D(w) = \left(
\begin{matrix}
1 & 0 & 0 & 0 & 0 & 0\\
0 & 0 & 1 & 0 & 0 & 0\\
0 & 1 & 0 & 0 & 1 & 0\\
0 & 0 & 0 & 1 & 0 & 1
%1 & 1 & 0\\
\end{matrix}
\right) \quad\quad 
\Phi^{-1}(D(w)) = 
\begin{ytableau}
6 & 5 & 3 & 1\\
6 & 5 & 3\\
6 & 5 & 2\\
6 & 4
\end{ytableau} = \pi \in \mathrm{ST}(4332,6)
}
$$
$L_4(w) = \lambda_1 = 4$ e.g. the subsequence $(1,2,4,4)$; $L_3(w) = \lambda_2 = 3$ e.g. the subsequence $(2,4,4)$; $L_2(w) = \lambda_3 = 3$ e.g. the subsequence $(3,4,4)$;  $L_1(w) = \lambda_4 = 2$ with the subsequence $(4,4)$. 
\end{example}

\begin{remark}
This theorem is an analog of Greene's theorem for RSK, see e.g. \cite[Ch.~7]{sta}.
\end{remark}

\begin{remark}
Probabilistic versions of these results were discussed in \cite{dy5}.
\end{remark}

\section{Monotonicity and bounds for descents enumeration functions}
For $\alpha =(\alpha_1, \ldots, \alpha_m) \in \mathbb{N}^m$, %with $|\alpha| = n$, 
let $D_{\alpha}(k, n, m)$ be the number of plane partitions $\pi \in \mathrm{PP}(k, n, m)$ with $\alpha_i$ descents of entry $i$ in $\pi$ for $i \in [1,m]$; alternatively entry $i$ is in $\alpha_i$ columns of $\pi$.

For $\alpha, \beta \in \mathbb{N}^m$ we write $\beta \succeq \alpha$ and say that $\beta$ {\it dominates} $\alpha$ if 
$$\sum_{i = 1}^{\ell} \beta_i \ge \sum_{i = 1}^{\ell} \alpha_i \quad\text{  for all $\ell \in [1,m]$.}$$

\begin{theorem}
The following properties hold:

(i) symmetry: $D_{\alpha}(k, n, m)$ does not change under permutations of $(\alpha_1, \ldots, \alpha_m)$

(ii) monotonicity: 
$$
D_{\alpha}(k, n, m) \ge D_{\beta}(k, n, m) \text{ for $\beta \succeq \alpha$ and $|\alpha| = |\beta|$.}
$$
\end{theorem}

\begin{proof}
By the definition of %dual Grothendieck polynomials 
$g_{\lambda}$ and Lemma~\ref{l1}, we have for $\rho= (k^{n})$
$$
\sum_{\alpha \in \mathbb{N}^m} D_{\alpha}(k, n, m)\, x^{\alpha} = g_{\rho}(1, x_1, \ldots, x_m) = s_{\rho}(1^{n}, x_1, \ldots, x_m).
$$
In particular, (i) becomes evident since $g_{\rho}$ is a symmetric polynomial. Let us expand the Schur polynomial using the branching formula
$$
s_{\rho}(1^{n}, x_1, \ldots, x_m) = \sum_{\lambda \subset \rho} s_{\lambda}(x_1,\ldots, x_m) s_{\rho/\lambda}(1^{n}) = \sum_{\lambda \subset \rho} \sum_{\alpha} K_{\lambda \alpha}\, x^{\alpha} s_{\rho/\lambda}(1^{n}),
$$
where $K_{\lambda \alpha}$ is the Kostka number (i.e. the coefficients in the monomial expansion of $s_{\lambda}$, or the number of SSYT of shape $\lambda$ of type $\alpha$). Combining these identities we obtain the following formula
$$
D_{\alpha}(k, n, m) = \sum_{\lambda \subset \rho} K_{\lambda \alpha}\, s_{\rho/\lambda}(1^{n}).
$$
%where $\lambda^c = (k - \lambda_{\ell}, \ldots, k - \lambda_1)$ is the complement of $\lambda$.
The item (ii) now follows from a similar inequality for Kostka numbers \cite{macdonald}, $K_{\lambda \alpha} \ge K_{\lambda \beta}$ for $\beta \succeq \alpha$. Indeed we get
$$
D_{\alpha}(k, n, m) = \sum_{\lambda \subset \rho} K_{\lambda \alpha}\, s_{\rho/\lambda}(1^{n}) \ge \sum_{\lambda \subset \rho} K_{\lambda \beta}\, s_{\rho/\lambda}(1^{n}) = D_{\beta}(k, n, m)
$$
as needed.
\end{proof}

\begin{corollary} Let $\alpha \in \mathbb{N}^m$ with $|\alpha| =N \le \min(k,m)$. The following inequalities hold
\begin{align}\label{inn}
\binom{n + N}{N} \le D_{\alpha}(k, n, m) \le n^{N}
\end{align}
%[give/discuss asymptotics]
\end{corollary}
\begin{proof}
From the monotonicity (ii) we have 
$$
D_{(N,0,\ldots, 0)}(k, n, m) \le D_{\alpha}(k, n, m) \le D_{(1^{N}, 0, \ldots, 0)}(k,n, m).
$$
Recall that $D_{(N,0,\ldots, 0)}(k, n, m)$ is the number of plane partitions with a single entry occupying $N \le k$ columns in the rectangle $(k^{n})$. This number is clearly $\binom{n + N}{N}$ which gives the lower bound. 

Next, note that $D_{(1^{N}, 0, \ldots, 0)}(k,n, m)$ is the number of plane partitions with entries $\{1, \ldots, N\}$ such that each entry occupies a single column. Note that each entry is uniquely determined by its lowest row position which can be any number from $1$ to $n$. Hence the total number of such plane partitions is at most $n^N$ which gives the upper bound.
\end{proof}
Note that the bounds do not depend on $k,m$. The lower bound is sharp %, and the upper bound is asymptotically sharp for ?. 
and the inequalities %\eqref{inn}
imply sharp asymptotics e.g. for $k, m \ge N$ and constant $n$
$$
\log D_{\alpha}(k, n, m) = N \log n + o(N) \text{ as } N\to \infty.
$$
Otherwise, e.g. for $N = \Theta(n)$ we get the bounds
$$
\Theta(N) \le \log D_{\alpha}(k, n, m) \le N \log n.
$$
Note also that there is another upper bound (the formula when $k \to \infty$ is unbounded is not difficult to derive using our bijection)
$$
D_{\alpha}(k, n, m) \le D_{\alpha}(\infty, n, m) = \prod_{i = 1}^m \binom{n + \alpha_i - 1}{\alpha_i} \le \binom{m n + N -1}{N}. %, \qquad  D_{n}(\infty, \ell, m) = \binom{\ell m + n - 1}{n}.
$$

\section*{Acknowledgements}
I am grateful to Askar Dzhumadil'daev, Igor Pak, and Pavlo Pylyavskyy for many helpful conversations. I am also grateful to the referees for careful reading of the text and many useful comments. 

%\newpage

\end{document}